\documentclass[11pt]{article}
\date{\vspace{-5ex}}
\usepackage{amsmath}
\usepackage{amsfonts}
\usepackage{amssymb}
\usepackage{amsthm}
\usepackage{breqn}
\usepackage{setspace}
\usepackage{fullpage}
\usepackage{enumitem}
\usepackage{bbold} 
\usepackage{comment}
\usepackage{hyperref}
\usepackage{bbm}
\usepackage{tikz}
\usepackage{enumitem}
\usepackage{mathtools} 
\usetikzlibrary{patterns,arrows,decorations.pathreplacing}

\newtheorem{theorem}{Theorem} 
\newtheorem{lemma}[theorem]{Lemma}
 
\newtheorem{definition}{Definition}
\newtheorem{claim}{Claim}

\usepackage{authblk}

\usepackage{authblk}

\begin{document}
\title{Wiener Index of Quadrangulation Graphs}

\author[1,2]{Ervin Gy\H{o}ri} 
\author[2]{Addisu Paulos}
\author[2]{Chuanqi Xiao}
\affil[1]{Alfr\'ed R\'enyi Institute of Mathematics, Budapest\par
\texttt{gyori.ervin@renyi.mta.hu}}
\affil[2]{Central European University, Budapest\par
\texttt{addisu_2004@yahoo.com, chuanqixm@gmail.com}}
\maketitle
\begin{abstract}
The Wiener index of a graph $G$, denoted $W(G)$, is the sum of the distances between all pairs of vertices in $G$. \' E. Czabarka, et al. conjectured that for an $n$-vertex, $n\geq 4$, simple quadrangulation graph $G$ 
\begin{equation*}W(G)\leq
\begin{cases}
\frac{1}{12}n^3+\frac{7}{6}n-2, &\text{ $n\equiv 0~(mod \ 2)$,}\\
\frac{1}{12}n^3+\frac{11}{12}n-1, &\text{ $n\equiv 1~(mod \ 2)$}.
\end{cases}
\end{equation*}
In this paper, we confirm this conjecture. 
\end{abstract}
\section{Introduction}
All graphs considered in this paper are finite, simple and connected. Let $G$ be a graph, then the vertex and edge sets of $G$ are denoted by $V(G)$ and $E(G)$ respectively. The Wiener index of the graph $G$ is denoted by $W(G)$ and is defined as,
\begin{equation*}
    W(G) = \sum_{\{u,v\} \subseteq V(G)} d_G(u,v),
\end{equation*} 
where $d_G(u,v)$ is the distance between the vertices $u$ and $v$ in the graph $G$.\\
Wiener index was first introduced by H. Wiener in 1947, while studying its correlations with boiling points of paraffin considering its molecular structure  \cite{1}. Since then, it has been one of the most frequently used topological indices in chemistry, as molecular structures are usually modelled as undirected graphs.\\
Obtaining sharp and asymptotically sharp bounds  and characterizing extremal structures are among wide varieties of previous and ongoing studies related to Wiener index.
The most basic upper bound of $W(G)$ states that, if $G$ is a connected graph of order $n$,
then
\begin{align*}
    W(G) \leq \frac{(n-1)n(n+1)}{6}
\end{align*} 
which is attained only when $G$ is a path \cite{6,7}. Many sharp or asymptotically sharp bounds on $W(G)$ in terms of other graph parameters are known, for instance, minimum degree \cite{8,9,10}, connectivity \cite{11,12}, edge-connectivity \cite{13,14} and maximum degree \cite{15}.\\
\' E. Czabarka, et al. \cite{17}, gave an asymptotic upper bound for the Wiener index of triangulation and quadrangulation graphs. Moreover, based on their constructions they proposed upper bound conjectures for the Wiener index of triangulation and quadrangulation graphs. Recently, D. Ghosh et al. \cite{22} confirmed their conjecture in the case of triangulation graphs. 

For quadrangulation graphs, \' E. Czabarka, et al. proved the following asymptotic upper bound. 
\begin{theorem}\cite{17}
Let $\kappa=\{2,3\}$, then there exist a constant $C$ such that
$$W(G)\leq \frac{1}{6\kappa}n^3+Cn^{\frac{5}{2}},$$
for every $\kappa$-connected simple quadrangulation G of order n.
\end{theorem}
In the same paper \cite{17}, based on the following constructions $Q_n$, see Figure \ref{qqq}, they conjectured that $W(Q_n)$ is the upper bound for the Wiener index of quadrangulation graphs.
\begin{equation*}W(Q_n)=
\begin{cases}
\frac{1}{12}n^3+\frac{7}{6}n-2, &\text{ $n\equiv 0~(mod \ 2)$,}\\
\frac{1}{12}n^3+\frac{11}{12}n-1, &\text{ $n\equiv 1~(mod \ 2)$}.
\end{cases}
\end{equation*}
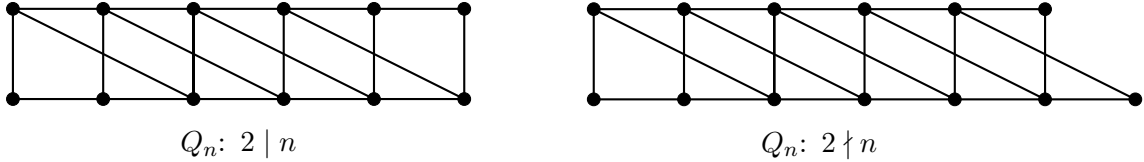
\begin{figure}[h]
\centering
\begin{tikzpicture}[scale=0.3]
\draw[fill=black](0,0)circle(8pt);
\draw[fill=black](4,0)circle(8pt);
\draw[fill=black](8,0)circle(8pt);
\draw[fill=black](12,0)circle(8pt);
\draw[fill=black](16,0)circle(8pt);
\draw[fill=black](20,0)circle(8pt);
\draw[fill=black](0,4)circle(8pt);
\draw[fill=black](4,4)circle(8pt);
\draw[fill=black](8,4)circle(8pt);
\draw[fill=black](12,4)circle(8pt);
\draw[fill=black](16,4)circle(8pt);
\draw[fill=black](20,4)circle(8pt);
\draw[thick](0,0)--(4,0)--(8,0)--(12,0)--(16,0)--(20,0);
\draw[thick](0,4)--(4,4)--(8,4)--(12,4)--(16,4)--(20,4);
\draw[thick](0,0)--(0,4)(8,0)--(8,4)(16,0)--(16,4)(4,0)--(4,4)(8,0)--(8,4)(12,0)--(12,4)(20,0)--(20,4);
\draw[thick](0,4)--(8,0);
\draw[thick](4,4)--(12,0);
\draw[thick](8,4)--(16,0);
\draw[thick](12,4)--(20,0);
\node at (10,-2){$Q_n\colon \ 2\mid n$};
\end{tikzpicture}\quad\quad\quad\quad
\begin{tikzpicture}[scale=0.3]
\draw[fill=black](0,0)circle(8pt);
\draw[fill=black](4,0)circle(8pt);
\draw[fill=black](8,0)circle(8pt);
\draw[fill=black](12,0)circle(8pt);
\draw[fill=black](16,0)circle(8pt);
\draw[fill=black](20,0)circle(8pt);
\draw[fill=black](24,0)circle(8pt);
\draw[fill=black](0,4)circle(8pt);
\draw[fill=black](4,4)circle(8pt);
\draw[fill=black](8,4)circle(8pt);
\draw[fill=black](12,4)circle(8pt);
\draw[fill=black](16,4)circle(8pt);
\draw[fill=black](20,4)circle(8pt);
\draw[thick](0,0)--(4,0)--(8,0)--(12,0)--(16,0)--(20,0)--(24,0);
\draw[thick](0,4)--(4,4)--(8,4)--(12,4)--(16,4)--(20,4);
\draw[thick](0,0)--(0,4)(8,0)--(8,4)(16,0)--(16,4)(4,0)--(4,4)(8,0)--(8,4)(12,0)--(12,4)(20,0)--(20,4);
\draw[thick](0,4)--(8,0);
\draw[thick](4,4)--(12,0);
\draw[thick](8,4)--(16,0);
\draw[thick](12,4)--(20,0);
\draw[thick](16,4)--(24,0);
\node at (10,-2){$Q_n\colon \ 2\nmid n$};
\end{tikzpicture}
\caption{Quadrangulations maximizing the Wiener index.}
\label{qqq}
\end{figure}

In this paper, we confirm their conjecture and have the following main result.
\begin{theorem}\label{main}
Let $G$ be a quadrangulation graph with $n\geq 4$ vertices. Then 
\begin{equation*}W(G)\leq
\begin{cases}
\frac{1}{12}n^3+\frac{7}{6}n-2, &\text{ $n\equiv 0~(mod \ 2)$,}\\
\frac{1}{12}n^3+\frac{11}{12}n-1, &\text{ $n\equiv 1~(mod \ 2)$}.
\end{cases}
\end{equation*}
\end{theorem}
To prove the statement, first we need some notations and preliminaries.
\section{Notations and Preliminaries}
Let $G$ be an $n$-vertex graph and $v\in V(G)$, the degree of $v$ in $G$ is denoted by $d_G(v)$, the set of vertices which are adjacent to $v$ by $N(v)$ and the minimum degree in $G$ by $\delta(G)$. Let $\emptyset\neq S\subset V(G)$, the status of the set $S$ is defined as
\begin{align*}
  \sigma_G(S)=\sum_{u\in V(G)}d_G(S,u), 
\end{align*}
where $d_G(S,u)$ is the distance of $u$ from $S$ in $G$, that is,
\begin{align*}
  d_G(S,u)=\min\{d_G(u,v)|v\in S\}.  
\end{align*}
In particular, if $S=\{v\}$, then the status of $S$ is denoted by $\sigma_G(v)$.
We may write $\sigma(S)$ rather than $\sigma_{G}(S)$, if the underlying graph $G$ is clear. Similarly $d(v)$ in place of $d_G(v)$. For the sake of simplicity, we use the term $k-$cycle instead of cycle of length $k$ and $k-$face instead of face of length $k$. Let $C$ be a cycle in $G$ which is embedded in the plane, the interior (exterior) of $C$ is the bounded (unbounded) part of $C$ excluding $C$.\\
The following Lemmas are needed to complete the proof of our main theorem.
\begin{lemma}
Let $G$ be an $n$-vertex quadrangulation graph, $n\geq 4$, then $\delta(G)$ is either 2 or 3.
\end{lemma}
\begin{proof}
Since each face is of length 4 and every edge is in two faces, then $4f=2e$. Using the Euler's formula, $e+2=n+f$, we get $e=2n-4$, thus, $\sum_{v\in V(G)}d(v)=4n-8$. Therefore, $\delta(G)\leq 3$. Since simple quadrangulation is 2-connected, there is no vertex of degree 1. Therefore, $\delta(G)$ is either 2 or 3. \end{proof}
\begin{definition}
A separating 4-cycle $S$ in a quadrangulation graph $G$ is a 4-cycle such that the deletion of $S$ from $G$ results in a disconnected graph. In other words, a separating 4-cycle is a 4-cycle which is not the
boundary of a face.
\end{definition}
\begin{lemma}\label{cd}\cite{21}
If $G$ is a quadrangulation with at least 6 vertices and no separating 4-cycle, then $G$ is 3-connected.
\end{lemma}
Let $G$ be a simple graph and $S$ be a nonempty subset of $V(G)$. Then we can partition vertices in $V(G)$ based on their distance from $S$. We call the set of vertices at the distance $i$ as the  \textit{$i-th$ level} with respect to $S$ and call the far most nonempty level \textit{terminal level} with respect to $S$. We have the following lemmas.
\begin{lemma}\label{lm1}
Let $G$ be a graph on $n+s$ vertices, $S$ be a set of vertices in $G$ such that $|S|=s$ and each of the non terminal levels with respect to $S$ contains at least 2 vertices. Then 
\begin{align*}
\sigma(S)\leq \begin{cases}
\frac{1}{4}(n^2+2n), &\text{ $2\mid n$,}\\
\frac{1}{4}(n^2+2n+1), &\text{ $2\nmid n$.}
\end{cases}
\end{align*}
\end{lemma}
\begin{proof}
Let $r$ denote the number of non terminal levels and $x_i$ be the number of vertices in the $i-th$ level for $i=1,2,\dots,r+1$. Thus $\sum_{i=1}^{r+1}x_i=n$. Since the terminal level contains at least 1 vertex and the other levels contain at least 2 vertices, we get that $r\leq \frac{n-1}{2}$ and 
\begin{align*}
\sigma(S)=&\sum_{i=1}^{r+1}ix_i=\sum_{i=1}^{r}ix_i+(r+1)x_{r+1}=\sum_{i=1}^{r}ix_i+(r+1)(n-\sum_{i=1}^{r}x_i)=(r+1)n+\sum_{i=1}^{r}(i-r-1)x_i\\&\leq (r+1)n-2\sum_{i=1}^{r}i=(r+1)n-r(r+1)   
\end{align*}
Let $f(r)=(r+1)n-r(r+1)$, we can see that $f(r)$ is maximized when $r=\frac{n-1}{2}$. Since $r$ is an integer, when $2\mid n$, $f(r)$ is maximized when $r=\frac{n-2}{2}$. Moreover, when $2\nmid n$, $f(r)$ is maximized obviously by $r=\frac{n-1}{2}$. Thus,
\begin{align*}
\sigma(S)\leq \begin{cases}
\frac{1}{4}(n^2+2n), &\text{$2\mid n$,}\\
\frac{1}{4}(n^2+2n+1), &\text{ $2\nmid n$.}
\end{cases}
\end{align*}
\end{proof}
\begin{lemma}\label{lm2}
Let $G$ be a graph on $n+s$ vertices, $S$ be a set of vertices in $G$ such that $|S|=s$ and each of the non terminal levels with respect to $S$ contains at least 2 vertices except the second level, which contains at least 3 vertices. Then 
\begin{align*}
\sigma(S)\leq \begin{cases}
\frac{1}{4}(n^2+8), &\text{ $2\mid n$,}\\
\frac{1}{4}(n^2+7), &\text{ $2\nmid n$.}
\end{cases}
\end{align*}
\end{lemma}
\begin{proof}
Assume that there are $r$ non terminal levels and $x_i$ be the number of vertices in the $i-th$ level for $i=1,2,\dots,r+1$, thus $\sum_{i=1}^{r+1}x_i=n$. Since the terminal level contains at least 1 vertex, the second level contains at least 3 vertices and the other levels contain at least 2 vertices,  we get $r\leq \frac{n-2}{2}$. Also we have that 
\begin{align*}
\sigma(S)=&\sum_{i=1}^{r+1}ix_i=\sum_{i=1}^{r}ix_i+(r+1)x_{r+1}=\sum_{i=1}^{r}ix_i+(r+1)(n-\sum_{i=1}^{r}x_i)=(r+1)n+\sum_{i=1}^{r}(i-r-1)x_i\\&\leq (r+1)n-2\sum_{i=1}^{r}i-(r-1)=(r+1)n-r(r+1)-r+1   
\end{align*}
Let $f(r)=(r+1)n-r(r+1)-r+1$, we can see that $f(r)$ is maximized when $r=\frac{n-2}{2}$. Similarly, since $r$ is an integer, when $2\mid n$, $f(r)$ is maximized when $r=\frac{n-2}{2}$. Moreover, when $2\nmid n$, $f(r)$ is maximized by $r=\frac{n-3}{2}$. Thus,
\begin{align*}
\sigma(S)\leq \begin{cases}
\frac{1}{4}(n^2+8), &\text{ $2\mid n$,}\\
\frac{1}{4}(n^2+7), &\text{ $2\nmid n$.}
\end{cases}
\end{align*}
\end{proof}
Due to similarity with the proofs given in the previous two lemmas, we omit the proof of the next lemma.
\begin{lemma}\label{lm9}
Let $G$ be an $n+s$-vertex graph and $S$ be a set of vertices in $G$ such that $|S|=s$. If each of the non terminal levels with respect to $S$ contains at least 3 vertices, then 
\begin{align*}
\sigma(S)\leq\frac{1}{6}(n^2+3n+2).
\end{align*}
\end{lemma}
\section{Proof of Theorem \ref{main}}
We proof Theorem \ref{main} by induction on the number of vertices $n$. As indicated in \cite{17} that Theorem \ref{main} holds for $n\leq 20$. Suppose that Theorem \ref{main} holds for all quadrangulation graphs with at most $n-1$ vertices. When $|V(G)|=n$, we distinguish 2 cases of the proof depending on the minimum degree of $G$.
\subsection*{Case 1: $\delta(G)=2$.}
Let $v$ be a vertex of degree 2 in $G$, there are 2 subcases, $n$ is even and $n$ is odd, separately.
\subsection*{Case 1.1: $2\mid n$.}
Deleting the vertex $v$, the resulting graph, $G-v$, is a quadrangulation graph on $n-1$ vertices. Obviously, $$W(G)\leq W(G-v)+\sigma(v).$$
By Lemma \ref{lm1}, since $2\nmid (n-1)$, $\sigma(v)\leq \frac{1}{4}\bigg((n-1)^2+2(n-1)+1\bigg)$ and by the induction hypothesis: $W(G-v)\leq \frac{1}{12}(n-1)^3+\frac{11}{12}(n-1)-1$.
Thus, 
$$W(G)\leq\bigg(\frac{1}{12}(n-1)^3+\frac{11}{12}(n-1)-1\bigg)+\frac{1}{4}\bigg((n-1)^2+2(n-1)+1\bigg)=\frac{1}{12}n^3+\frac{7}{6}n-2$$ 
and we are done.
\subsection*{Case 1.2: $2\nmid n$.}
Here we consider two subcases based on the number of vertices in the second level with respect to $S=\{v\}$.
\subsubsection*{Case 1.2.1: The second level contains at least 3 vertices.}
Similarly as Case 1.1, after deleting the vertex $v$, we get an $(n-1)$-vertex quadrangulation graph, $G-v$. Since $2\mid (n-1)$, by the induction hypothesis, we get $W(G-v)\leq \frac{1}{12}(n-1)^3+\frac{7}{6}(n-1)-2$. And since the second level contains at least 3 vertices, by Lemma \ref{lm2}, $\sigma(v)\leq \frac{1}{4}\bigg((n-1)^2+8\bigg)$.
Thus, $$W(G)\leq W(G-v)+\sigma(v)\leq \bigg(\frac{1}{12}(n-1)^3+\frac{7}{6}(n-1)-2\bigg)+\frac{1}{4}\bigg((n-1)^2+8\bigg)=\frac{1}{12}n^3+\frac{11}{12}n-1.$$
\subsubsection*{Case 1.2.2: The second level contains 2 vertices.}
Let $N(v)=\{x_1,x_2\}$, $x_3$ and $x_4$ be vertices in $G$ such that $vx_1x_3x_2v$ and $vx_1x_4x_2v$ are two 4-faces sharing the path $x_1vx_2$, see Figure \ref{19} $(a)$. The vertices $x_3$ and $x_4$ are in the second level with respect to $v$. Thus, $d_G(x_1)=d_G(x_2)=3$. When $n\geq 7$ (actually, Theorem \ref{main} holds for $n\leq 20$) we have cherries $x_3z_1x_4$ and $x_3z_2x_4$ such that $x_3z_1x_4x_1x_3$ and $x_3z_2x_4x_2x_3$ are 4-faces, for distinct vertices $z_1$ and $z_2$ in $G$. 
\begin{figure}[h]
\centering
\begin{tikzpicture}[scale=0.25]
\draw[fill=black](0,0)circle(8pt);
\draw[fill=black](-10,0)circle(8pt);
\draw[fill=black](10,0)circle(8pt);
\draw[fill=black](0,6)circle(8pt);
\draw[fill=black](0,-6)circle(8pt);
\draw[fill=black](0,0)circle(8pt);
\draw[fill=black](0,-12)circle(8pt);
\draw[fill=black](0,12)circle(8pt);
\draw[thick](-10,0)--(0,12)--(10,0)--(0,-12)--(-10,0)--(0,6)--(10,0)--(0,-6)--(-10,0)(0,-6)--(0,0)--(0,6);
\node at (-12,0){$x_3$};
\node at (12,0){$x_4$};
\node at (0,7){$x_2$};
\node at (0,-7){$x_1$};
\node at (0,13){$z_2$};
\node at (0,-13){$z_1$};
\node at (1,0){$v$};
\node at (0,-16){$(a)$};
\end{tikzpicture}\quad\quad
\begin{tikzpicture}[scale=0.25]
\draw[fill=black](0,0)circle(8pt);
\draw[fill=black](-11,0)circle(8pt);
\draw[fill=black](11,0)circle(8pt);
\draw[fill=black](0,0)circle(8pt);
\draw[fill=black](0,-12)circle(8pt);
\draw[fill=black](0,12)circle(8pt);
\draw[thick](-11,0)--(0,12)--(11,0)--(0,-12)(-11,0)--(0,0)--(11,0)(0,-12)--(-11,0);
\node at (-13,0){$x_3$};
\node at (13,0){$x_4$};
\node at (0,13){$z_2$};
\node at (0,-13){$z_1$};
\node at (0,-1){$x$};
\node at (0,-16){$(b)$};
\end{tikzpicture}
\caption{}
\label{19}
\end{figure}

Contracting edges $x_1v$ and $x_2v$ to a vertex $x$, see Figure \ref{19} $(b)$, results an $(n-2)$-vertex quadrangulation graph, say $G'$. Notice that in the graph $G'$, for any two  vertices $t_1, t_2\in V(G')\backslash\{x\} $, $d_{G'}(t_1,t_2)=d_G(t_1,t_2)$. But for any vertex $t\in V(G')\backslash \{x\}$, $d_G(t,v)=d_{G'}(t,x)+1$. We know that,
\begin{align*}
W(G)&=\sum_{\{u,w\}\subseteq V(G)}d_G(u,w)\\&=\sum_{u,w\in V(G)\backslash\{x_1,x_2\}}d_G(u,w)+\sum_{u\in V(G)\backslash\{x_2\}}d_G(u,x_1)+\sum_{u\in V(G)\backslash\{x_1\}}d_G(u,x_2)+d_G(x_1,x_2) \\
\end{align*}
and
\begin{align*}
\sum_{{u,w}\in V(G)\backslash\{x_1,x_2\}}d_G(u,w)&=\sum_{{u,w}\in V(G)\backslash\{x_1,x_2,v\}}d_G(u,w)+\sum_{u\in V(G)\backslash\{x_1,x_2\}}d_G(u,v)\\&=\sum_{{u,w}\in V(G')\backslash\{x\}}d_{G'}(u,w)+\sum_{u\in V(G')\backslash\{x\}}(d_{G'}(u,x)+1)\\&=\sum_{{u,w}\in V(G')}d_{G'}(u,w)+(n-3)\\&=W(G')+(n-3)
\end{align*}
Now we estimate $\sum\limits_{u\in V(G)\backslash\{x_2\}}d_G(u,x_1)+\sum\limits_{u\in V(G)\backslash\{x_1\}}d_G(u,x_2)+2$. Consider the set $S=\{x_3,x_4,z_1,z_2\}$ and the levels of vertices in $S'=V(G)\backslash \{x_1,x_2,x_3,x_4,v,z_1,z_2\}$ with respect to $S$. Notice that $|S'|=n-7$. Sine $2\nmid n$, then $2\mid (n-7)$. By Lemma \ref{lm1}, $$\sigma(S)\leq \frac{1}{4}\bigg((n-7)^2+2(n-7)\bigg).$$
Let $u\in S'$, then $d_G(u,x_1)+d_G(u,x_2)\leq 2d_G(S,u)+4$. Thus, 
\begin{align*}
\sum_{u\in S'}d_G(u,x_1)+\sum_{u\in S'}d_G(u,x_2)&\leq 2\sum_{u\in S'}d_G(S,u)+4(n-7)=2\sigma(S)+4(n-7)\\
&=\frac{1}{2}\bigg((n-7)^2+2(n-7)\bigg)+4(n-7).
\end{align*}
It can be checked that the sum of the distance of $x_i$, $i\in[2]$, to each of the vertices in $\{x_3,x_4,z_1,z_2,v\}$ is 12. Hence, 
\begin{align*}
\sum_{u\in V(G)\backslash\{x_2\}}d_G(u,x_1)+\sum_{u\in V(G)\backslash\{x_1\}}d_G(u,x_2)+2\leq \frac{1}{2}\bigg((n-7)^2+2(n-7)\bigg)+4(n-7)+14
\end{align*}
By the induction hypothesis, we get
\begin{align*}
W(G)&=W(G^')+(n-3)+\sum_{u\in V(G)\backslash\{x_2\}}d_G(u,x_1)+\sum_{u\in V(G)\backslash\{x_1\}}d_G(u,x_2)+2\\
&\leq \bigg(W(G')+(n-3)\bigg)+\bigg(\frac{1}{2}((n-7)^2+2(n-7))+4(n-7)+14\bigg) \\&\leq \bigg(\frac{1}{12}(n-2)^3+\frac{11}{12}(n-2)-1+(n-3)\bigg)+\bigg(\frac{1}{2}((n-7)^2+2(n-7))+4(n-7)+14\bigg)\\&=\frac{1}{12}n^3+\frac{11}{12}n-3
\end{align*}
Hence we are done in this case.
\subsection*{Case 2: $\delta(G)=3$.}
Let $v$ be a vertex of degree 3, $N(v)=\{v_1, v_2,v_3\}$ and $v_4,v_5,v_6$ be vertices in $G$ such that $v_1vv_3v_4v_1$, $v_2vv_3v_5v_2$ and $v_1vv_2v_6v_1$ are 4-faces, see Figure \ref{min3}. Notice that $v_4,v_5$ and $v_6$ are distinct, otherwise, $G$ contains a degree 2 vertex. Moreover, at least two of $e_1=\{v_1,v_5\}, e_2=\{v_2,v_4\}$ and $e_3=\{v_3,v_6\}$ are not in $E(G)$. If none of the three exists in $G$, we call the associated vertex $v$ as a "\textit{good vertex}". Observe that, deleting $v$ and adding one of these missed edges in $G$ result a quadrangulation graph with $n-1$ vertices.
Suppose that $e_1=\{v_1,v_5\}\notin E(G)$ and $G_1$ is the quadrangulation graph obtained form $G$ by deleting $v$ and adding $e_1$. Since this will decrease the distance of some pairs of vertices, we define the total sum of the decrease distance due to this operation on $G$ as following,
$$\text{dec}(G,G_1)=\sum_{u,w\in V(G)\backslash\{v\}}\bigg(d_G(u,w)-d_{G_1}(u,w)\bigg).$$

We need the following claim to complete the proof of this case.
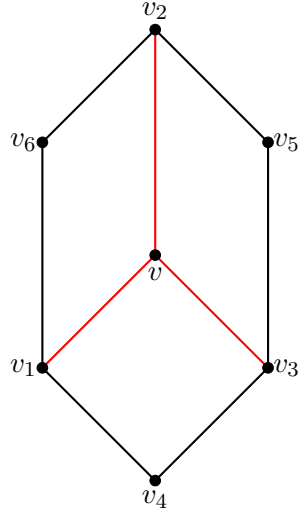
\begin{figure}[h]
\centering
\begin{tikzpicture}[scale=0.25]
\node at (0,-13){$v_4$};
\node at (7,-6){$v_3$};
\node at (7,6){$v_5$};
\node at (0,13){$v_2$};
\node at (-7,6){$v_6$};
\node at (-7,-6){$v_1$};
\node at (0,-1){$v$};
\draw[thick](0,-12)--(6,-6)--(6,6)--(0,12)--(-6,6)--(-6,6)--(-6,-6)--(0,-12);
\draw[red,thick](0,0)--(0,12)(0,0)--(-6,-6)(0,0)--(6,-6);
\draw[fill=black](0,0)circle(8pt);
\draw[fill=black](0,12)circle(8pt);
\draw[fill=black](0,-12)circle(8pt);
\draw[fill=black](-6,6)circle(8pt);
\draw[fill=black](6,-6)circle(8pt);
\draw[fill=black](-6,-6)circle(8pt);
\draw[fill=black](6,6)circle(8pt);
\end{tikzpicture}
\label{min3}
\caption{Structure about a degree 3 vertex of quadrangulation graph with $\delta(G)=3$}
\end{figure}

\begin{claim}\label{cl1}
Let $v$ be a good vertex and $G_i$ be a quadrangulation graph obtained from $G$ by deleting $v$ and adding the edge $e_i$ for $i\in[3]$. Then
\begin{displaymath}
\min_{i\in[3]} \{\text{dec}(G,G_i)\} \leq \frac{(n-1)^2}{18}
\end{displaymath}
\end{claim}

\begin{proof} Without loss of generality, assume that $G_1$ gives the minimum total sum of the decrease distance.
Let $S=\{v_1,v_2,v_3,v_4,v_5,v_6\}$ and $S_i=S\backslash\{v_i\}$, $i\in[6]$. Define $A_i=\{x\in V(G_1)\backslash S_i|\ d_{G_1}(v_i,x)<d_{G_1}(z,x), z\in S_i\}$ and $i\in[6]$. Let $A$ be the set of the remaining vertices, which means that $A=V(G_1)-\bigcup\limits_{i\in[6]}A_i$. Thus, $V(G_1)=A_1\cup A_2\cup\dots\cup A_6\cup A$. Now we show that, for any pair of vertices $\{u,w\}$ in $G_1$ with shortest path that must across the edge $e_1$, then one vertex is in $A_1$ and the other is in $A_5$. If not, there are 4 possibilities,
\begin{enumerate}
\item suppose one of $u$ or $w$ is not in $A_1$, $A_5$ and $A$, without loss of generality, let $w\in A_i$, $i\in\{2,3,4,6\}$. It can be easily checked that there exists a shortest path between $w$ and $u$ which does not across the edge $e_1$, then we find a contradiction.

Now suppose one of $u$ and $w$ is in $A$, there are two possibilities, without loss of generality, let $u\in A$  :
\item $d_{G_1}(v_i,u)=d_{G_1}(v_j,u)$, where $d(v_i,v_j)$ is odd, $i,j\in [6]$, this implies that there exists an odd cycle in $G_1$, which contradicts to the fact that every quadrangulation graphs are free of odd cycles. 
\item  $d_{G_1}(v_i,u)=d_{G_1}(v_j,u)$, where $d(v_i,v_j)$ is even, $i,j\in [6]$. It can be easily checked that there exists a shortest path between $w$ and $u$ which does not across the edge $e_1$, then we find a contradiction.
\item Now, there are only one case left, suppose $u$ and $w$ are both in $A_1$ or in $A_5$. In this case, it is easy to see that the shortest path can never use the edge $e_1$, which is a contradiction.
\end{enumerate}
Therefore, a pair of vertices with shortest paths must across the edge $e_1$ are those pairs that one is in $A_1$, the other one is in $A_5$. Since $\text{dec}(G,G_1)=\min\limits_{i\in[3]} \{\text{dec}(G,G_i)\}$, $\text{dec}(G,G_1)$ is maximized when $|A_1|=|A_2|=\dots=|A_6|=\frac{n-1}{6}$. Notice that for such pair of vertices, the decreased distance is at most 2. Therefore, $$\text{dec}(G,G_1)\leq 2\bigg(\frac{n-1}{6}\bigg)\bigg(\frac{n-1}{6}\bigg)=\frac{(n-1)^2}{18}.$$
\end{proof}

Now we continue the proof of the case by considering two subcases based on the existence of separating 4-cycle in $G$.
\subsubsection*{Case 2.1: No separating 4-cycle in $G$.} 
Since there is no separating 4-cycle in $G$, by Lemma \ref{cd}, $G$ is 3-connected, then each degree 3 vertex is a good vertex. Take a degree 3 vertex $v\in V(G)$, without loss of generality, deleting $v$ and adding $e_1$, denoted $G_1$, gives the minimum decrease distance sum, which is less than $\frac{(n-1)^2}{18}$ based on Claim \ref{cl1}. Also since $G$ is 3-connected, each of the non terminal level with respect to a vertex $v$ contains at least 3 vertices, otherwise, there exists a 2-vertex cut in $G$. By Lemma \ref{lm9}, we have
$$\sigma_G(v)\leq\frac{1}{6}\bigg((n-1)^2+3(n-1)+2\bigg).$$
By the induction hypothesis, we get
\begin{align*}
W(G)&\leq W(G_1)+\sigma_G(v)+\text{dec}(G,G_1)\\
&\leq W(G_1)+\sigma_G(v)+\frac{(n-1)^2}{18}\\
&\leq \bigg(\frac{1}{12}(n-1)^3+\frac{7}{6}(n-1)-2\bigg)+\frac{1}{6}\bigg((n-1)^2+3(n-1)+2\bigg)+\frac{(n-1)^2}{18}\\&=\frac{n^3}{12}-\frac{n^2}{36}+\frac{53n}{36}-\frac{115}{36}
\end{align*}
It can be checked that $\frac{n^3}{12}-\frac{n^2}{36}+\frac{53n}{36}-\frac{115}{36}<\frac{1}{12}n^3+\frac{11}{12}n-1$ for $n\geq 15$, so the induction is settled in this case.
\begin{definition}
A \textit{minimum separating 4-cycle} in $G$ is a separating 4-cycle which contains no separating 4-cycle interior.
\end{definition}
\subsubsection*{Case 2.2: G contains a separating 4-cycle.}
Let $S=\{z_1,z_2,z_3,z_4\}$ be a minimum separating 4-cycle in $G$ with minimum number of vertices in the interior. Let  $x$ and $n-x-4$ be the number of vertices of the interior and exterior of $S$, respectively. Clearly, $x\geq 4$ and $n-x-4\geq 4$, otherwise, $G$ contains a vertex with degree 2. Removing the interior $x$ vertices of $S$ results in a quadrangulation graph, say $G_{n-x}$ on $n-x$ vertices. Removing the exterior $n-x-4$ vertices of $S$ results in a 3-connecetd quadrangulation graph, say $G_{x+4}$ on $x+4$ vertices. Obviously, we have
$$W(G)\leq W(G_{x+4})+ W(G_{n-x})-8+\sum_{w\in V(G_{n-x})\backslash {S}}\sum_{u\in v(G_{x+4})\backslash {S}}d(u,w),$$
here $W(S)$ be double counted results in $-8$, 
and 
\begin{align*}
\sum_{w\in V(G_{n-x})\backslash {S}}\sum_{u\in v(G_{x+4})\backslash {S}}d(u,w)\leq x\sigma_{G_{n-x}}(S)+(n-x-4)\bigg(\text{max}\{\sigma_{G_{x+4}}(z_i)~|~i\in [4]\}-4\bigg),
\end{align*}
here $\sum\limits_{j\neq i,~i,j\in [4]}d(z_i,z_j)$ be included in $\text{max}\{\sigma_{G_{x+4}}(z_i)~|~i\in [4]\}$ results in $-4$ .

Since $G_{x+4}$ is 3-connected, also by Lemma \ref{lm9}, we have $$\text{max}\{\sigma_{G_{x+4}}(z_i)~|~i\in [4]\}\leq \frac{1}{6}\bigg((x+3)^2+3(x+3)+2\bigg).$$
Thus, by the induction hypothesis, Lemmas \ref{lm1} and \ref{lm9}, we get
\begin{align*}
W(G)&\leq W(G_{x+4})+ W(G_{n-x})-8+\sum_{w\in V(G_{n-x})\backslash {S}}\sum_{u\in v(G_{x+4})\backslash {S}}d(u,w)\\
&\leq \bigg(\frac{1}{12}(x+4)^3+\frac{7}{6}(x+4)-2\bigg)+\bigg(\frac{1}{12}(n-x)^3+\frac{7}{6}(n-x)-2\bigg)-8\\
&+\frac{x}{4}
\bigg((n-x-4)^2+2(n-x-4)+1\bigg)+(n-x-4)\bigg(\frac{1}{6}\bigg((x+3)^2+3(x+3)+2\bigg)-4\bigg)\\
&=\frac{n^3}{12}-\frac{nx^2}{12}+\frac{n}{2}+\frac{x^3}{12}+\frac{x^2}{3}+\frac{11x}{12}+\frac{2}{3}
\end{align*}
It can be checked that the inequality $\frac{n^3}{12}-\frac{nx^2}{12}+\frac{n}{2}+\frac{x^3}{12}+\frac{x^2}{3}+\frac{11x}{12}+\frac{2}{3}\leq \frac{1}{12}n^3+\frac{11}{12}n-1$ holds for $x\geq 4$, except when $x=4$ and $n=9$ which implies that $G_{n-x}\setminus S$ contains only one degree 2 vertex, this contradicts the fact that $\delta (G)=3$.
And the proof of the Theorem \ref{main} is complete.

\end{document}